\documentclass[11pt]{amsart}

\usepackage{amsthm, amsmath, amsfonts}
\usepackage{mathtools}
\usepackage{enumerate}
\usepackage[left]{lineno}

\usepackage{graphicx}
\usepackage{pgfplots}
\pgfplotsset{compat=1.12}
\usepackage{pgfplotstable}
\usepackage{subcaption}
\usepackage{hyperref}

\usepackage{fontawesome5}

\allowdisplaybreaks

\usepackage{a4wide}
\usepackage{todonotes}
\usepackage{tikz}
\usetikzlibrary{arrows,backgrounds,calc}
\usetikzlibrary{math,arrows,backgrounds}
\usetikzlibrary{decorations.pathreplacing}
\usetikzlibrary{fit}
\usepackage{amsthm, amsmath, amsfonts}
\usepackage{mathtools}

\newtheorem{proposition}{Proposition}

\newtheorem{definition}{Definition}
\newtheorem{theorem}{Theorem}
\newtheorem*{theorem*}{Theorem}
\newtheorem{lemma}{Lemma}
\newtheorem{corollary}{Corollary}

\newtheorem*{main}{Theorem~\ref{thm:main}}

\theoremstyle{remark}
\newtheorem*{remark}{Remark}

\title[]{Enumeration of Generalized Dyck Paths Based on the Height of Down-Steps Modulo $k$}

\author[C. Heuberger \and S. J. Selkirk \and S. Wagner]{Clemens Heuberger \and Sarah J. Selkirk \and Stephan Wagner}

\address[Clemens Heuberger, Sarah J. Selkirk] {Institut
  f\"ur Mathematik, Alpen-Adria-Uni\-ver\-si\-t\"at Klagenfurt,
  Universit\"atsstra\ss e 65--67, 9020 Klagenfurt, Austria. }
\email{\href{mailto:clemens.heuberger@aau.at}{clemens.heuberger (at) aau.at}, \href{mailto:sarah.selkirk@aau.at}{sarah.selkirk (at) aau.at}}
\address[Stephan Wagner] {Department of Mathematics, Uppsala Universitet, Box 480, 751 06 Uppsala, Sweden.}
\email{\href{mailto:stephan.wagner@math.uu.se}{stephan.wagner (at) math.uu.se}}
\thanks{The research of C.~Heuberger and S.~J.~Selkirk is partially supported by the
  Austrian Science Fund (FWF): P~28466-N35, \emph{Analytic Combinatorics: Digits, Automata and Trees} and
  Austrian Science Fund (FWF): DOC 78. S.~Wagner is supported by the Knut and Alice Wallenberg Foundation.}

\keywords{Dyck path, generating function, Lagrange inversion, bijection, Fuss--Catalan number, Raney number}
\subjclass[2010]{05A15}

\newcommand{\drawlatticepath}[1]{
  \tikzmath{
      \x = 0;
      \y = 0;
      for \step in {#1}{
        {
          \draw[thin, ->] (\x, \y) -- (\x + 1, \y + \step);
        };
        \x = \x + 1;
        \y = \y + \step;
      };
    };
}

\newcommand{\up}{
  \begin{tikzpicture}[scale = 0.3]
    \draw[-, thick] (0, 0) to (1, 1); 
  \end{tikzpicture}
}

\newcommand{\down}{
  \begin{tikzpicture}[scale = 0.3]
    \draw[-, thick] (0, 1) to (1/2, 0); 
  \end{tikzpicture}
}

\newcommand{\xbar}{
    \mathsf{\overline{x}}
}

\begin{document}

\begin{abstract}
    For fixed non-negative integers $k$, $t$, and $n$, with $t < k$, a $k_t$-Dyck path of length $(k+1)n$ is a 
    lattice path that starts at $(0, 0)$, ends at $((k+1)n, 0)$, stays weakly above the line $y = -t$, and consists of steps from the step-set 
    $\{(1, 1), (1, -k)\}$.
    We enumerate the family of $k_t$-Dyck paths by considering the number of down-steps at a height of $i$ modulo $k$. Given a tuple 
    $(a_1, a_2, \ldots, a_k)$ we find an exact enumeration formula for the number of $k_t$-Dyck paths of length $(k+1)n$ with $a_i$ down-steps at 
    a height of $i$ modulo $k$, $1 \leq i \leq k$.  The proofs given are done via bijective means or with generating functions. 
\end{abstract}

\maketitle

\section{Introduction}

For a fixed positive integer $k$, a \emph{$k$-Dyck path} with length $(k+1)n$ is a lattice path consisting of steps $\{(1, 1), (1, -k)\}$ 
which starts at $(0, 0)$, ends at $((k+1)n, 0)$, and stays (weakly) above the $x$-axis. This is a natural extension of Dyck paths. The 
family of $k$-Dyck paths (or their reverse, with step-set $\{(1, k), (1, -1)\}$) has been the subject of a number of papers (see for example 
\cite{Cameron:2016:Returns, Hilton-Pedersen:1991:catalan-generalization, Josuat-Kim:2016:Dyck-tilings, muhle-kalli:2014:strip-decomposition-dyck-path, 
 Xin-Zhang:2019:sweep-maps}).
The family of $k$-Dyck paths is enumerated by the Fuss--Catalan numbers 
\begin{equation*}
    \frac{1}{kn+1}\binom{(k+1)n}{n},
\end{equation*}
where references to OEIS sequences~\cite{OEIS:2022} for specific values of $k$ can be found in the second row of Table~\ref{tab:sequences}. Bijections have been established between them and several other 
combinatorial objects, see for example~\cite{Heubach-Li-Mansour:2008:k-Catalan} for a list of such objects (referred to as $k$-Catalan structures). 

In~\cite{Gu-Prodinger-Wagner:2010:k-plane-trees} Gu, Prodinger and Wagner studied the family of $k$-plane trees: labelled plane trees where 
each vertex is labelled $i$ with 
$i \in \{1, 2, \ldots, k\}$ such that the sum of labels along any edge is at most $k+1$, with the root labelled $k$. Here they established a 
relationship between these trees and a modified version of a $k$-Dyck paths, which have since been named $k_t$-Dyck paths, whose definition is 
given below.  

\begin{definition}[$k_t$-Dyck paths]
    For fixed non-negative integers $k$, $t$, and $n$, with $0 \leq t < k$, a \emph{$k_t$-Dyck path} of length $(k+1)n$ is a lattice path 
    consisting of $n$ steps of type $(1, -k)$ and $kn$ steps of type $(1, 1)$, that starts at $(0, 0)$, ends at $((k+1)n, 0)$, and stays 
    weakly above the line $y = -t$. 
\end{definition}

The number of $k_t$-Dyck paths of length $(k+1)n$ was shown to be equal to the Raney numbers (see~\cite{Selkirk:2019:MSc}), 
\begin{equation}\label{eq:raney}
    \frac{t+1}{(k+1)n+t+1}\binom{(k+1)n+t+1}{n},
\end{equation}
which when specialized to $t = 0$ is equal to the sequence of Fuss--Catalan numbers. The Raney numbers enumerate tuples of $k$-ary 
trees \cite{Selkirk:2019:MSc}, $k$-plane trees \cite{Gu-Prodinger-Wagner:2010:k-plane-trees}, types of planar embeddings 
\cite{Beagley-Drube:2015:Raney}, binary matrices used in coding theory \cite{Asinowski-Hackl-Selkirk:2020:downstep-statistics}, 
and most recently threshold sequences and Motzkin-like paths~\cite{Rusu:2021:Raney-numbers}.

\begin{table}[ht]
    \begin{tabular}{| c | c | c | c | c | c | c | c |}
        \hline
        {$t$} \textbackslash {$k$} & $1$ & $2$ & $3$ & $4$ & $5$ & $6$ & $7$\\
        \hline
        $0$ & \href{https://oeis.org/A000108}{A000108} & \href{https://oeis.org/A001764}{A001764} & \href{https://oeis.org/A002293}{A002293} & \href{https://oeis.org/A002294}{A002294}
        & \href{https://oeis.org/A002295}{A002295} & \href{https://oeis.org/A002296}{A002296} & \href{https://oeis.org/A007556}{A007556}\\
        \hline
        $1$ & X & \href{https://oeis.org/A006013}{A006013} & \href{https://oeis.org/A069271}{A069271} & \href{https://oeis.org/A118969}{A118969} 
        & \href{https://oeis.org/A212071}{A212071} & \href{https://oeis.org/A233832}{A233832} & \href{https://oeis.org/A234461}{A234461} \\
        \hline
        $2$ & X & X & \href{https://oeis.org/A006632}{A006632} & \href{https://oeis.org/A118970}{A118970} & \href{https://oeis.org/A212072}{A212072} 
        & \href{https://oeis.org/A233833}{A233833} & \href{https://oeis.org/A234462}{A234462}\\
        \hline
    \end{tabular}
    \caption{References to OEIS sequences~\cite{OEIS:2022} of equation~\eqref{eq:raney} for fixed values of $k$ and $t$.}
    \label{tab:sequences}
\end{table}

The family of $k_t$-Dyck paths can also be seen as lattice paths on a rectangular grid from $(0, 0)$ to $(kn, n)$, consisting of steps $(1, 0)$
and $(0, 1)$, where the paths never go above the line $ky = x + t$. Here the $(1, 0)$ steps are up-steps and the $(0, 1)$ steps are down-steps 
in our model, respectively. The case where $t = 0$ has been studied in~\cite{Banderier-Wallner:rational-slope:2019, Duchon:2000:generalized-Dyck,
Heubach-Li-Mansour:2008:k-Catalan, Hilton-Pedersen:1991:catalan-generalization, Rogers:schroeder-triangle:1977},
for example. 

In this paper, we will study the enumeration of $k_t$-Dyck paths based on the number of down-steps $(1, -k)$ at given heights modulo $k$. In
the rectangular grid model, this is the number of $(0, 1)$-steps with $x$-coordinates congruent to a given number modulo $k$. 
We will make use of the following notation to avoid excessive repetition of the same lengthy phrase. 
\begin{definition}
    Let $k$, $t$, and $n$ be non-negative integers with $0 \leq t < k$. For each $1 \leq i \leq k$ let $a_{i}$ 
    be a non-negative integer, where $a_1 + \cdots + a_k = n$. Then we define $\mathcal{K}_{k, t}^{n}(a_1, \ldots a_{k})$ 
    to be the set of all $k_t$-Dyck paths of length $(k+1)n$ with $a_i$ down-steps $(1, -k)$ such that the endpoints of the 
    down-steps are at a height of $i$ modulo $k$ for $1\leq i\leq k$. 
\end{definition}
As usual, we denote the number of elements in $\mathcal{K}_{k, t}^{n}(a_1, \ldots a_{k})$ by 
$\big|\mathcal{K}_{k, t}^{n}(a_1, \ldots a_{k})\big|$.

\begin{figure}[h]
    \begin{tikzpicture}[scale=0.2]
        \draw[help lines] (0,-1) grid (12,7);
        \draw[thick, ->] (-0.25, 0) -- (12.5, 0);
        \draw[thick, ->] (0, -1.25) -- (0, 7.5);
        \drawlatticepath{1, 1, 1, -3, 1, 1, 1, 1, -3, 1, 1, -3}
    \end{tikzpicture}\quad
    \begin{tikzpicture}[scale=0.2]
        \draw[help lines] (0,-1) grid (12,7);
        \draw[thick, ->] (-0.25, 0) -- (12.5, 0);
        \draw[thick, ->] (0, -1.25) -- (0, 7.5);
        \drawlatticepath{1, 1, 1, 1, -3, 1, 1, -3, 1, 1, 1, -3}
    \end{tikzpicture}\quad
    \begin{tikzpicture}[scale=0.2]
        \draw[help lines] (0,-1) grid (12,7);
        \draw[thick, ->] (-0.25, 0) -- (12.5, 0);
        \draw[thick, ->] (0, -1.25) -- (0, 7.5);
        \drawlatticepath{1, 1, 1, 1, -3, 1, 1, 1, 1, 1, -3, -3}
    \end{tikzpicture}\quad
    \begin{tikzpicture}[scale=0.2]
        \draw[help lines] (0,-1) grid (12,7);
        \draw[thick, ->] (-0.25, 0) -- (12.5, 0);
        \draw[thick, ->] (0, -1.25) -- (0, 7.5);
        \drawlatticepath{1, 1, 1, 1, 1, 1, 1, -3, 1, 1, -3, -3}
    \end{tikzpicture}\quad
    \begin{tikzpicture}[scale=0.2]
        \draw[help lines] (0,-1) grid (12,7);
        \draw[thick, ->] (-0.25, 0) -- (12.5, 0);
        \draw[thick, ->] (0, -1.25) -- (0, 7.5);
        \drawlatticepath{1, 1, 1, 1, 1, 1, -3, 1, -3, 1, 1, -3}
    \end{tikzpicture}
    \caption{The five $3_1$-Dyck paths in $\mathcal{K}_{3, 1}^3(1, 0, 2)$ for $n = 3$, $k = 3$, and $t = 1$.}
\end{figure}
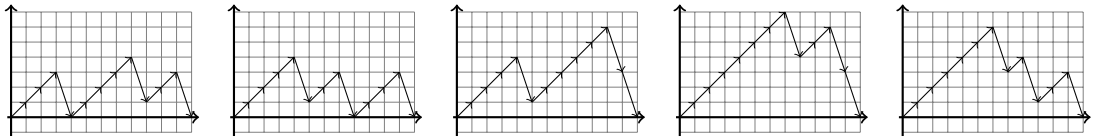
Okoth and Wagner have studied a similar statistic in~\cite{Okoth-Wagner:tbc:k-plane}, where they determined the enumeration of $k$-plane trees based on the number 
of vertices labelled $i$, as well as a similar enumeration of non-crossing trees. However, the bijection between $k$-plane trees and $k_t$-Dyck 
paths does not preserve this statistic, and the counting formula is different for the $k$-plane trees statistic, and thus 
one cannot derive a formula for $|\mathcal{K}_{k, t}^n(a_1, \ldots, a_k)|$ from this relationship.

In~\cite{Burstein:2020:peaks-mod-k}, Burstein considers the distribution of peak heights modulo $k$ and double descents in $k$-Dyck paths, 
which is closely related to this work since $k_0$-Dyck paths are $k$-Dyck paths, and a peak is an up-step followed by a down-step, so the height of peaks has 
a relationship with the height of down-steps. 

In this paper, we will make use of bijective and generating function methods to show the following main result along with related results. 

\begin{main}
    Let $k$, $n$ and $t$ be non-negative integers with $0 \leq t \leq k-1$. For each $1 \leq i \leq k$ let $a_{i}$ 
    be a non-negative integer, where $a_1 + \cdots + a_k = n$. Then the number 
    of $k_t$-Dyck paths of length $(k+1)n$ with $a_i$ down-steps at a height of $i$ modulo $k$, $1 \leq i \leq k$, is
    \begin{align*}
        \big|\mathcal{K}_{k, t}^{n}(a_1, \ldots a_{k})\big| = &\frac{a_{k-t}+\cdots+a_{k}}{n(n+1)}
        \prod_{i = k-t}^{k}\binom{n+a_i}{a_i}\prod_{i=1}^{k-t-1}\binom{n+a_i-1}{a_i}.
    \end{align*} 
\end{main}

Thereafter we will establish relationships between some special cases of the set 
$\mathcal{K}_{k, t}^{n}(a_1, \ldots a_{k})$ and other combinatorial objects such 
as peaks and valleys in Dyck paths, Dyck paths themselves, 
and the sum of the indices of the down-steps in all valleys in Dyck paths.

\section{Enumeration of $k_t$-Dyck paths according to the heights of down-steps modulo $k$}

We begin by enumerating a special case of $k_t$-Dyck path based on height of down-steps modulo $k$, when $t = k-1$, and will later use 
this to enumerate the general case. The proof of the special case is done bijectively, and the remaining cases are proved via generating 
function methods. 

\begin{lemma}\label{lem:wag-comb-interp}
    When $t = k-1$, $\big|\mathcal{K}_{k, t}^{n}(a_1, \ldots a_{k})\big|$ is symmetric in $a_1, \ldots, a_{k}$. That is,
    if $\overline{(a_1, \ldots, a_{k})}$ is a permutation of $(a_1, \ldots, a_{k})$, then
    \begin{equation*}
        \big|\mathcal{K}_{k, k-1}^{n}(a_1, \ldots, a_{k})\big| = \big|\mathcal{K}_{k, k-1}^{n}\overline{(a_1, \ldots, a_{k})}\big|.
    \end{equation*}
\end{lemma}
\begin{proof}
    When $k=1$ it is clear that the lemma holds, since any permutation of $a_1$ is trivial. Therefore we assume that $k \geq 2$ 
    for the remainder of this proof. 

    Firstly, note that because down-steps $(1, -k)$ do not contribute to height modulo $k$, the height of a down-step modulo $k$ 
    is entirely dependent on the number of up-steps $(1, 1)$ that occur before it in the path. The central concept of this proof 
    is that we can swap two (maximal) sets of down-steps at different heights modulo $k$. A simple case of this is that one can shift 
    down-steps at a height of $k$ modulo $k$ to a height of $k-1$ modulo $k$ (and vice versa) by shifting 
    the down-steps at a height of $k$ modulo $k$ one up-step to the left (vice versa: right) in the path. This is allowed because 
    by definition ($t < k$), a down-step (or a consecutive sequence thereof) at a height of $k$ modulo $k$ must end at a height of 
    $y = 0$ or above and cannot occur at the beginning of a path, and thus the shifted down-step(s) will end at a height of $-1$ 
    or above, and thus be (weakly) above $y = -t = -k+1$. Similarly, a down-step at a height of $k-1$ modulo $k$ will 
    not occur directly at the end of a path, and thus it will be followed by at least one up-step which makes the shift one up-step to the right 
    permissible.   
    Note that 
    consecutive sequences of down-steps that are shifted are all shifted one up-step to the left (vice versa: right) and thus remain 
    consecutive. Exchanging $a_{k-1}$ and $a_k$ (simultaneously) using these shifts as demonstrated in Figure~\ref{fig:shift-for-permutation} 
    is a bijection, and thus  
    $\big|\mathcal{K}_{k, k-1}^{n}(a_1, \ldots, a_{k-2}, a_{k-1}, a_{k})\big| = \big|\mathcal{K}_{k, k-1}^{n}(a_1, \ldots, a_{k-2}, a_{k}, a_{k-1})\big|$.
    \begin{figure}[ht]
        \begin{tikzpicture}[scale = 0.35]
            \draw[help lines] (0,-2) grid (12,4);
            \draw[thick, ->] (-0.25, 0) -- (12.5, 0);
            \draw[thick, ->] (0, -2.25) -- (0, 4.5);
            \drawlatticepath{1, -3, 1, 1, 1, 1, 1, -3, 1, 1, -3, 1}
            \draw[->, thick, red] (7, 3) to (8, 0);
            \draw[->, thick, blue] (10, 2) to (11, -1);
        \end{tikzpicture}\raisebox{2.5em}{$\rightarrow$}
        \begin{tikzpicture}[scale = 0.35]
            \draw[help lines] (0,-2) grid (12,4);
            \draw[thick, ->] (-0.25, 0) -- (12.5, 0);
            \draw[thick, ->] (0, -2.25) -- (0, 4.5);
            \drawlatticepath{1, -3, 1, 1, 1, 1, -3, 1, 1, 1, 1, -3}
            \draw[->, thick, red] (6, 2) to (7, -1);
            \draw[->, thick, dotted, red] (7, 3) to (8, 0);
            \draw[->, thick, dotted] (6, 2) to (7, 3);
            \draw[->, thick, dotted, blue] (10, 2) to (11, -1);
            \draw[->, thick, dotted] (11, -1) to (12, 0);
            \draw[->, thick, blue] (11, 3) to (12, 0);
        \end{tikzpicture}\raisebox{2.5em}{$\rightarrow$}
        \begin{tikzpicture}[scale = 0.35]
            \draw[help lines] (0,-2) grid (12,4);
            \draw[thick, ->] (-0.25, 0) -- (12.5, 0);
            \draw[thick, ->] (0, -2.25) -- (0, 4.5);
            \drawlatticepath{1, -3, 1, 1, 1, 1, -3, 1, 1, 1, 1, -3}
            \draw[->, thick, red] (6, 2) to (7, -1);
            \draw[->, thick, blue] (11, 3) to (12, 0);
        \end{tikzpicture}
        \caption{A $3_2$-Dyck path of length $12$. The down-steps $(1, -3)$ at a height of $3$ modulo $3$ (marked in red) are 
        `exchanged' with those at a height of $2$ modulo $3$ (marked in blue).}
        \label{fig:shift-for-permutation}
    \end{figure}
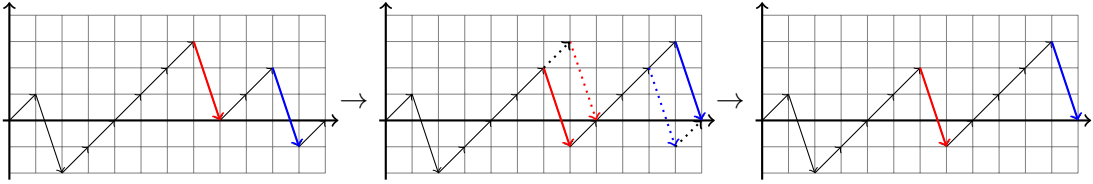
    
    This shifting can be generalised: 
    \begin{itemize}
        \item Let $1 \leq j < i \leq k$. To shift down-steps from a height of $i$ modulo $k$ to $j$ modulo $k$, down-steps should 
        be shifted $i-j$ up-steps to the left.
        \item Let $1 \leq i < j \leq k$. To shift down-steps from a height of $i$ modulo $k$ to $j$ modulo $k$, down-steps should 
        be shifted $j-i$ up-steps to the right.
    \end{itemize}
    The lowest possible endpoint for a down-step (or a consecutive sequence thereof) at a height of $i$ modulo $k$ is $y = -k+i$, 
    therefore exchanges of $a_i$ to any 
    other $a_j$ are possible as long as a shift to the left of $i$ (or more) positions does not take place (otherwise the path would not stay 
    weakly above $y = -(k-1)$). Indeed, this is prevented since any left shift is $i-j$ up-steps and from the condition on left 
    shifts, $0 < i-j \leq i-1$. Lastly, there will always be (at least) $i-j$ up-steps to the left of a down-step (or a consecutive 
    sequence thereof) at a height of $i$ modulo $k$ since $i - j \leq i$. 
    Similarly, a shift to the right is restricted by the number of up-steps to the right of a down-step. 
    Since the total number of up-steps is $kn$ and thus congruent to $k$ modulo $k$, there are a minimum of $k-i$ up-steps after
    a down-step (or a consecutive sequence thereof) at a height of $i$ modulo $k$, and thus exchanges of $a_i$ to any other $a_j$ 
    are possible as long as a shift 
    to the right of $k-i+1$ (or more) places does not take place, and the condition for steps to the right ensures that $j-i \leq k-i$.
    
    Since left and right shifts have been shown to be legal, any pair $a_i$ and $a_j$ can be simultaneously exchanged uniquely 
    (which steps shift left and right is clearly defined), meaning 
    \begin{equation*}
    \big|\mathcal{K}_{k, k-1}^{n}(a_1, \ldots, a_{i}, \ldots, a_j, \ldots, a_{k})\big| = \big|\mathcal{K}_{k, k-1}^{n}(a_1, \ldots, a_{j}, \ldots, a_i, \ldots, a_{k})\big|.
    \end{equation*}
    Simultaneous exchange can be thought of as composition of the two shifts, which are independent of the positioning of other 
    down-steps, keeping track of the original heights of the steps to avoid shifting one height of steps twice.
    Since any permutation of $(a_1, \ldots, a_k)$ can be obtained via a finite number of exchanges of two elements, we have proved 
    the result. 
\end{proof}
\begin{remark}
    The symmetry described in Lemma~\ref{lem:wag-comb-interp} and the proof thereof can be applied to the more general case within 
    $(a_{k-t}, \ldots, a_{k})$ and $(a_1, \ldots, a_{k-t-1})$, of which $t = k-1$ is a special case.
\end{remark}

\begin{proposition}\label{prop:wag-gen-func}
    For $k$ a positive integer,
    \begin{equation}\label{prop:eqn}
        \big|\mathcal{K}_{k, k-1}^{n}(a_1, \ldots a_{k})\big| = \frac{1}{n+1}\prod_{i = 1}^{k}\binom{n+a_i}{a_i}.
    \end{equation}
\end{proposition}
\begin{proof}
    Consider any $k_{k-1}$-Dyck path, and consider the weak right-to-left minima of the path (in this paper, a weak right-to-left minimum is an end-point of a down-step
    $(1, -k)$ that is no higher than any points that are further right). See 
    Figure~\ref{fig:decomp} for a depiction of this. 
    Let $M$ be the set of these points. For each point in $M$, we consider the first return to the $x$-axis that comes after 
    it (i.e., the first point to the 
    right of it that lies on the $x$-axis; for a weak right-to-left minimum on the $x$-axis, that is just the point itself). 
    Let $R$ be the set of these first returns.

    We cut the $k_{k-1}$-Dyck path at the points in $R$ to divide it into a number of segments (remark: for $k=1$, this is 
    the standard decomposition of Dyck paths into a sequence of arches). Each of them contains exactly one point of 
    $M$ either inside or at its right end (we do not count the left end even if it also belongs to $M$). This is because given 
    two consecutive right-to-left minima at heights $i$ and $h$ where $1-k \leq i \leq 0$ and $i \leq h \leq 0$, there must be at 
    least $k$ up-steps between the two minima for the down-step ending at height $h$ to occur. Thus either one of the $k$ 
    up-steps has an endpoint on the $x$-axis or both right-to-left minima are on the $x$-axis.
    If a segment 
    contains a point in $M$ whose $y$-coordinate is $j-k$, then its shape is
    \begin{itemize}
        \item $j$ steps up, followed by a (vertically shifted) $k_{k-1}$-Dyck path, followed by a down-step $(1, -k)$ and $k-j$ steps up.
    \end{itemize}
    
    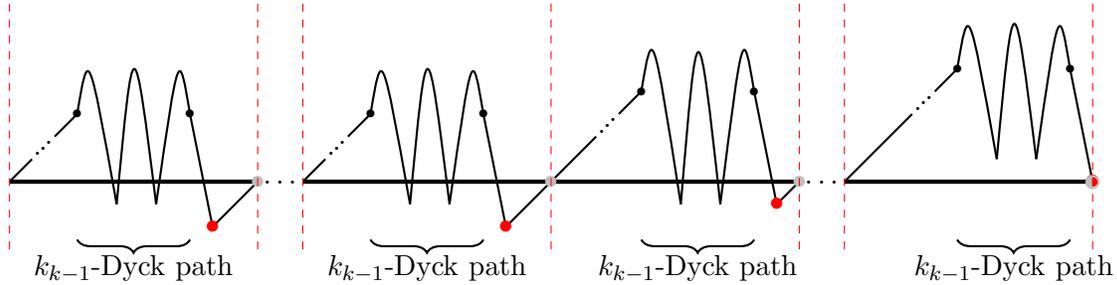
\begin{figure}[h]
        \begin{tikzpicture}[scale = 0.3]
            \draw [-, ultra thick] (-13, 0) to (-2, 0);

            \node at (-0.9, 0) {\ldots};
            
            \draw [-, ultra thick] (0, 0) to (22, 0);

            \node at (23.1, 0) {\ldots};
            
            \draw [-, ultra thick] (24, 0) to (35, 0);

            \draw [-, thick] (0-13, 0) to (1-13, 1);
            \node at (1.25-13, 1.25) {.};
            \node at (1.5-13, 1.5) {.};
            \node at (1.75-13, 1.75) {.};
            \draw [-, thick] (2-13, 2) to (3-13, 3);
        
            \draw [-, thick] (3-13, 3) .. controls (3.5-13, 6) .. (4.75-13, -1);
            \draw [-, thick] (4.75-13, -1) .. controls (5.5-13, 7) .. (6.5-13, -1);
            \draw [-, thick] (6.5-13, -1) .. controls (7.5-13, 6) .. (8-13, 3);
        
            \draw [-, thick] (8-13, 3) to (9-13, -2);
            \draw [-, thick] (9-13, -2) to (11-13, 0);
            \node[red] at (-4, -2) {$\bullet$};
            \node[lightgray] at (-2, 0) {$\bullet$};
        
            \node at (3-13, 3) {\tiny{$\bullet$}};
            \node at (8-13, 3) {\tiny{$\bullet$}};

            \draw [thick, decorate, decoration={brace, amplitude=5pt, mirror, raise=4pt}]
            (-10cm, -2.1) to node[below,yshift=-0.2cm] {$k_{k-1}$-Dyck path}(-5cm, -2.1);

            %%%%%%%%%%%%%%%%        
            \draw [-, thick] (0, 0) to (1, 1);
            \node at (1.25, 1.25) {.};
            \node at (1.5, 1.5) {.};
            \node at (1.75, 1.75) {.};
            \draw [-, thick] (2, 2) to (3, 3);
        
            \draw [-, thick] (3, 3) .. controls (3.5, 6) .. (4.75, -1);
            \draw [-, thick] (4.75, -1) .. controls (5.5, 7) .. (6.5, -1);
            \draw [-, thick] (6.5, -1) .. controls (7.5, 6) .. (8, 3);
        
            \draw [-, thick] (8, 3) to (9, -2);
            \draw [-, thick] (9, -2) to (11, 0);
            \node[red] at (9, -2) {$\bullet$};
        
            \node at (3, 3) {\tiny{$\bullet$}};
            \node at (8, 3) {\tiny{$\bullet$}};

            \draw [thick, decorate, decoration={brace, amplitude=5pt, mirror, raise=4pt}]
            (3cm, -2.1) to node[below,yshift=-0.2cm] {$k_{k-1}$-Dyck path}(8cm, -2.1);

            %%%%%%%%%%%%%%%%
            \draw [-, thick] (11, 0) to (13, 2);
            \node at (13.25, 2.25) {.};
            \node at (13.5, 2.5) {.};
            \node at (13.75, 2.75) {.};
            \draw [-, thick] (14, 3) to (15, 4);
        
            \draw [-, thick] (15, 4) .. controls (15.5, 7) .. (16.75, -1);
            \draw [-, thick] (16.75, -1) .. controls (17.5, 8) .. (18.5, -1);
            \draw [-, thick] (18.5, -1) .. controls (19.5, 7) .. (20, 4);
        
            \draw [-, thick] (20, 4) to (21, -1);
            \draw [-, thick] (21, -1) to (22, 0);
            \node[red] at (21, -1) {$\bullet$};
            \node[lightgray] at (22, 0) {$\bullet$};
            \node[lightgray] at (11, 0) {$\bullet$};

            \node at (15, 4) {\tiny{$\bullet$}};
            \node at (20, 4) {\tiny{$\bullet$}};

            \draw [thick, decorate, decoration={brace, amplitude=5pt, mirror, raise=4pt}]
            (15cm, -2.1) to node[below,yshift=-0.2cm] {$k_{k-1}$-Dyck path}(20cm, -2.1);
            %%%%%%%%%%%%%%%%%

            \draw [-, thick] (24, 0) to (27, 3);
            \node at (27.25, 3.25) {.};
            \node at (27.5, 3.5) {.};
            \node at (27.75, 3.75) {.};
            \draw [-, thick] (28, 4) to (29, 5);

            \draw [-, thick] (29, 5) .. controls (29.5, 8) .. (30.75, 1);
            \draw [-, thick] (30.75, 1) .. controls (31.5, 9) .. (32.5, 1);
            \draw [-, thick] (32.5, 1) .. controls (33.5, 8) .. (34, 5);

            \draw [-, thick] (34, 5) to (35, 0);
            \node[red] at (35, 0) {\large{$\bullet$}};
            \node[lightgray] at (35, 0) {\scalebox{0.5}{\faAdjust}};

            \node at (29, 5) {\tiny{$\bullet$}};
            \node at (34, 5) {\tiny{$\bullet$}};

            \draw [thick, decorate, decoration={brace, amplitude=5pt, mirror, raise=4pt}]
            (29cm, -2.1) to node[below,yshift=-0.2cm] {$k_{k-1}$-Dyck path}(34cm, -2.1);

            \draw[-, dashed, red] (-13,-3) to (-13, 8);
            \draw[-, dashed, red] (-2,-3) to (-2, 8);
            \draw[-, dashed, red] (-0,-3) to (-0, 8);
            \draw[-, dashed, red] (11,-3) to (11, 8);
            \draw[-, dashed, red] (22,-3) to (22, 8);
            \draw[-, dashed, red] (24,-3) to (24, 8);
            \draw[-, dashed, red] (35,-3) to (35, 8);
        \end{tikzpicture}
        \caption{The decomposition of a $k_{k-1}$-Dyck path into segments based on weak right-to-left minima $M$ (marked in red)
        and the corresponding returns to the $x$-axis $R$ (marked in gray).}
        \label{fig:decomp}
    \end{figure}
    So we have a decomposition into shorter $k_{k-1}$-Dyck paths (and some additional steps), which can be translated 
    to a generating function identity. 
    
    Let $F_{k-1}(z, x_1, \ldots, x_{k})$ be the generating function for $k_{k-1}$-Dyck paths where $z$ represents $(1, -k)$ steps, 
    and $x_j$ represents down-steps at height $j$ modulo $k$. For ease of notation, let $\xbar = (z, x_1, \ldots, x_k)$. The weak 
    right-to-left minima's $y$-coordinates read from 
    left to right are a (possibly empty) sequence of $1-k$ values, followed by a (possibly empty) sequences of $2-k$ values, and continues  
    in this way until it ends with a (possibly empty) sequence of $0$ values.    
    Let $\mathcal{K}_{k, t}$ be the class of $k_t$-Dyck paths, \up be a $(1, 1)$ step, and \down be a $(1, -k)$ step. Then using the 
    notation of \cite[Chapter I]{Flajolet-Sedgewick:ta:analy}, the symbolic decomposition of $\mathcal{K}_{k, k-1}$ according to 
    right-to-left minima is given by 
    \begin{equation*}
        \mathsf{SEQ}\big(\{\up\} \times \mathcal{K}_{k, k-1} \times \{\down\} \times (\{\up\})^{k-1}\big)\times \cdots \times
        \mathsf{SEQ}\big(\{\up\}^k \times \mathcal{K}_{k, k-1} \times \{\down\} \times \{\up\}^0\}\big) 
    \end{equation*}
    From this decomposition of $k_{k-1}$-Dyck paths, we obtain the functional equation 
    \begin{equation}\label{eq:sym-gen-func}
        F_{k-1}(\xbar) = \prod_{i=1}^{k}\frac{1}{1 - zx_i F_{k-1}(\xbar)}.
    \end{equation}
    Although the $k_{k-1}$-Dyck paths in the decomposition are shifted vertically, we can consider all height shifts of them 
    as equal due to Lemma~\ref{lem:wag-comb-interp}: the shifted paths are symmetric and thus have the same generating function.

    We now apply Lagrange Inversion to obtain the coefficients of the generating function~\eqref{eq:sym-gen-func} 
    which will prove the formula in~\eqref{prop:eqn}. 
    Note that the Lagrange Inversion Formula~\cite[Appendix A.6]{Flajolet-Sedgewick:ta:analy} states that if a power series $y$ 
    satisfies the equation $y(z) = z\Phi(y(z))$, where $\Phi$ is a formal power series with $\Phi(0) \neq 0$, then 
    \begin{equation*}
        [z^n]y(z) = \frac{1}{n}[t^{n-1}]\Phi(t)^n
    \end{equation*}
    for $n\ge 1$. To apply this, set $F_{k-1}^{*}(\xbar) = zF_{k-1}(\xbar)$. Then for $\Phi(t) = \prod_{i=1}^{k}\frac{1}{1 - x_i t}$, 
    \begin{equation*}
        F_{k-1}^{*}(\xbar) = z\prod_{i=1}^{k}\frac{1}{1 - x_i F_{k-1}^{*}(\xbar)} = z\Phi(F_{k-1}^{*}(\xbar)),
    \end{equation*}
    which is amenable to Lagrange Inversion, resulting in 
    \begin{equation*}
        [z^n]F_{k-1}(\xbar) = [z^{n+1}]F_{k-1}^{*}(\xbar) 
        = \frac{1}{n+1}[t^n]\prod_{i=1}^{k}\frac{1}{(1 - x_it)^{n+1}},
    \end{equation*}
    and thus for $n = a_1 + a_2 + \cdots + a_k$, 
    \begin{equation*}
        [z^nx_1^{a_1}x_2^{a_2}\cdots x_k^{a_k}]F_{k-1}(\xbar)
        %= \frac{1}{n+1}[t^nx_1^{a_1}x_2^{a_2}\cdots x_k^{a_k}]\prod_{i=1}^{k}\frac{1}{(1 - x_it)^{n+1}}
        = \frac{1}{n+1}[(tx_1)^{a_1}(tx_2)^{a_2}\cdots (tx_k)^{a_k}]\prod_{i=1}^{k}\frac{1}{(1 - x_it)^{n+1}}.
    \end{equation*}
    Thus we can consider just $[x_1^{a_1}\cdots x_{k}^{a_{k}}]$ 
    to obtain more specific information involving $a_i$-values while implicitly keeping track of the $t^n$ term. 
    It is clear that 
    \begin{equation*}
        [x_j^{a_j}]\prod_{i=1}^{k}\frac{1}{(1 - x_it)^{n+1}} = \binom{n+a_j}{a_j} t^{a_j}\prod_{\substack{i=1\\ i \neq j}}^{k}\frac{1}{(1 - x_it)^{n+1}},
    \end{equation*}
    and thus 
    \begin{equation*}
        [z^nx_1^{a_1}\cdots x_{k}^{a_{k}}]F_{k-1}(\xbar) = \frac{1}{n+1}\prod_{i=1}^{k}\binom{n+a_i}{a_i}.
    \end{equation*}
\end{proof}

\begin{theorem}\label{thm:main}
    Let $k$, $n$ and $t$ be non-negative integers with $0 \leq t \leq k-1$. For each $1 \leq i \leq k$ let $a_{i}$ 
    be a non-negative integer, where $a_1 + \cdots + a_k = n$. Then the number 
    of $k_t$-Dyck paths of length $(k+1)n$ with $a_i$ down-steps at a height of $i$ modulo $k$, $1 \leq i \leq k$, is
    \begin{equation}\label{eq:main}
        \big|\mathcal{K}_{k, t}^{n}(a_1, \ldots a_{k})\big| = \frac{a_{k-t}+\cdots+a_{k}}{n(n+1)}
        \prod_{i = k-t}^{k}\binom{n+a_i}{a_i}\prod_{i=1}^{k-t-1}\binom{n+a_i-1}{a_i}.
    \end{equation} 
\end{theorem}
\begin{proof}
    We decompose $k_{t}$-Dyck paths in an analogous way to that of the proof of Proposition~\ref{prop:wag-gen-func}: by determining 
    the weak right-to-left minima, and splitting the path into segments according to the first return of the path to the $x$-axis 
    after the minima. The decomposition for the case $t = k-1$ can be seen in Figure~\ref{fig:decomp}. Analogously, 
    using $F_t(\xbar)$ to denote the generating function for $k_t$-Dyck paths where $z$ counts the number
    of $(1, -k)$ steps and $x_j$ counts the number of $x_j$ steps at height $j$ modulo $k$, we obtain the functional equation
    \begin{equation}\label{eq:general-func-eq}
        F_t(\xbar) = \prod_{i = k-t}^{k}\frac{1}{1-zx_i F_{k-1}(\xbar)}.
    \end{equation}
    We now apply Lagrange Inversion in the form as given in \cite[Equation 2.1.1]{Gessel:2016:lagrange}, 
    which states that 
    if $y(z) = z\Phi(y(z))$ where $\Phi$ is a formal power series with $\Phi(0) \neq 0$, then for a power 
    series $g$, we have that 
    \begin{equation*}
        [z^n]g(y(z)) = \frac{1}{n}[s^{n-1}]g'(s)\Phi(s)^n, \qquad \text{for } n \geq 1.
    \end{equation*}
    This is satisfied in the following way: Set $y(z) = F_{k-1}^{*}(\xbar)\coloneqq zF_{k-1}(\xbar)$ and setting 
    \begin{equation*}
        g(s) = \prod_{i = k-t}^{k}\frac{1}{1- x_i s},
    \end{equation*}
    we see that $F_t(\xbar) = g(F_{k-1}^{*}(\xbar))$. Additionally, $\Phi$ is the same as in Proposition~\ref{prop:wag-gen-func}.
    We differentiate $g$ to obtain
    \begin{equation*}
        g'(y) = \sum_{j = k-t}^{k}\frac{x_j}{1-x_j y}\prod_{i=k-t}^{k}\frac{1}{1-x_iy},
    \end{equation*}
    and thus we find that the coefficients are
    \begin{equation}\label{eq:coeff-lagrange-2}
        [z^{n}x_1^{a_1}\cdots x_{k}^{a_{k}}]g(y) = \frac{1}{n}[s^{n-1}x_1^{a_1}\cdots x_{k}^{a_{k}}]\sum_{j = k-t}^{k}\frac{x_j}{1-x_j s}\prod_{i=k-t}^{k}\frac{1}{(1-x_is)^{n+1}}
        \prod_{i=1}^{k-t-1}\frac{1}{(1-x_is)^n}. 
    \end{equation}
    In the equations below we can ignore the $s^{n-1}$ term after the second equality and thenceforth set $s = 1$. This is 
    because we once again note that 
    \begin{equation*}
        [s^{n-1}x_1^{a_1}\cdots x_{j-1}^{a_{j-1}}x_{j}^{a_{j}-1}x_{j+1}^{a_{j+1}}\cdots x_{k}^{a_{k}}] = [(x_1s)^{a_1}\cdots (x_{j-1}s)^{a_{j-1}}(x_{j}s)^{a_{j}-1}(x_{j+1}s)^{a_{j+1}}\cdots (x_{k}s)^{a_{k}}], 
    \end{equation*}
    since the sum of all $a_i$ terms for $1 \leq i \leq k$ is $n-1$. By~\eqref{eq:coeff-lagrange-2}, we have
    \begin{align*}
        n[z^{n}&x_1^{a_1}\cdots x_{k}^{a_{k}}]g(y) \\
        & = [s^{n-1}x_1^{a_1}\cdots x_{k}^{a_{k}}]\sum_{j = k-t}^{k}\frac{x_j}{1-x_j s}\prod_{i=k-t}^{k}\frac{1}{(1-x_is)^{n+1}}
        \prod_{i=1}^{k-t-1}\frac{1}{(1-x_is)^n}\\
        & = \sum_{j=k-t}^{k}[x_1^{a_1}\cdots x_{j-1}^{a_{j-1}}x_{j}^{a_{j}-1}x_{j+1}^{a_{j+1}}\cdots x_{k}^{a_{k}}]\frac{1}{(1-x_j)^{n+2}}\prod_{\substack{i=k-t \\ i \neq j}}^{k}\frac{1}{(1-x_i)^{n+1}}
        \prod_{i=1}^{k-t-1}\frac{1}{(1-x_i)^n}\\
        & = \sum_{j=k-t}^{k}\binom{n+a_j}{a_j-1}\prod_{\substack{i=k-t \\ i \neq j}}^{k}\binom{n+a_i}{a_i}\prod_{i=1}^{k-t-1}\binom{n+a_i-1}{a_i}\\
        & = \sum_{j=k-t}^{k}\frac{a_j}{n+1}\binom{n+a_j}{a_j}\prod_{\substack{i=k-t \\ i \neq j}}^{k}\binom{n+a_i}{a_i}\prod_{i=1}^{k-t-1}\binom{n+a_i-1}{a_i}\\
        & = \frac{1}{n+1}\prod_{i=k-t}^{k}\binom{n+a_i}{a_i}\prod_{i=1}^{k-t-1}\binom{n+a_i-1}{a_i}\sum_{j=k-t}^{k}a_j.
    \end{align*}
    Altogether for $n \geq 1$ this gives
    \begin{equation*}
        [z^{n}x_1^{a_1}\cdots x_{k}^{a_{k}}]F_{t}(\xbar) = \frac{\sum_{j=k-t}^{k}a_j}{n(n+1)}\prod_{i=k-t}^{k}\binom{n+a_i}{a_i}\prod_{i=1}^{k-t-1}\binom{n+a_i-1}{a_i},
    \end{equation*}
    which is what we wanted to prove. 
\end{proof}

The result obtained is demonstrated in Table~\ref{tab:permutations-and-values} for $k=3$ and $n=4$. Note that the last row is independent 
of permutations of $a_1a_2a_3$, as proved in Lemma~\ref{lem:wag-comb-interp}. Additionally, summing across rows, it can be checked that the Raney 
number~\eqref{eq:raney} for $k=3$, $n = 4$, and $t$ as given in the corresponding row is obtained. 
\begin{table}[ht]

    \begin{tabular}{|c||c|c|c|c|c|c|c|c|c|c|c|c|c|c|c|}
        \hline
        $a_1a_2a_3$ & 004 & 040 & 400 & 013 & 031 & 103 & 130 & 301 & 310 & 022 & 202 & 220 & 112 & 121 & 211 \\
        \hline
        \hline
        $t = 0$ & 14 & 0 & 0 & 21 & 5 & 21 & 0 & 5 & 0 & 15 & 15 & 0 & 24 & 10 & 10 \\ 
        \hline
        $t = 1$ & 14 & 14 & 0 & 35 & 35 & 21 & 21 & 5 & 5 & 45 & 15 & 15 & 45 & 45 & 25 \\ 
        \hline
        $t = 2$ & 14 & 14 & 14 & 35 & 35 & 35 & 35 & 35 & 35 & 45 & 45 & 45 & 75 & 75 & 75 \\ 
        \hline
    \end{tabular}
    \caption{The number of $k_t$-Dyck paths of length $16$ with $k = 3$ which have $a_i$ down-steps at a height of $i$ modulo $k$.}
    \label{tab:permutations-and-values}
\end{table}

\begin{remark}[On Schur-positivity]
    In Lemma~\ref{lem:wag-comb-interp}, it was proven that $\big|\mathcal{K}_{k, k-1}^{n}(a_1, \ldots a_{k})\big|$ is symmetric 
    in $a_1,\ldots,a_k$. Equivalently, the polynomial
    \begin{equation*}
    [z^n] F_{k-1}(\xbar)
    \end{equation*}
    is a symmetric polynomial in $x_1,\ldots,x_k$. As the proof of Proposition~\ref{prop:wag-gen-func} shows, this polynomial can be expressed as
    \begin{equation*}
    [z^n] F_{k-1}(\xbar) = \frac{1}{n+1} [t^n] \prod_{i=1}^k \frac{1}{(1-x_i t)^{n+1}}.
    \end{equation*}
    Generating functions of this form are familiar from the theory of Schur polynomials. Indeed, with $s_{\lambda}$ denoting the Schur polynomial 
    associated with a partition $\lambda$, we have (see e.g.~\cite[Corollary 8.16]{Aigner:2007:course-enumeration})
    \begin{equation*}
    \sum_{\lambda} s_{\lambda}(x_1,x_2,\ldots) s_{\lambda}(y_1,y_2,\ldots) = \prod_{i,j \geq 1} \frac{1}{1-x_i y_j}.\end{equation*}
    Setting $x_{k+1}=x_{k+2} = \cdots = y_{n+2} = y_{n+3} = \cdots = 0$ and $y_1 = y_2 = \cdots = y_{n+1} = t$, we find that
    \begin{equation*}
    [z^n] F_{k-1}(\xbar) = \frac{1}{n+1} [t^n] \sum_{\lambda} s_{\lambda}(x_1,x_2,\ldots,x_k) s_{\lambda}(t,t,\ldots,t).
    \end{equation*}
    Note here that $[t^n] s_{\lambda}(t,t,\ldots,t)$ is non-zero only when $\lambda$ is a partition of $n$. In this case, it is exactly the 
    number of semistandard Young tableaux of shape $\lambda$ with entries in $\{1,2,\ldots,n+1\}$. This representation shows in particular that 
the polynomial $[z^n] F_{k-1}(\xbar)$ is Schur-positive.
\end{remark}

\begin{theorem}\label{thm:total}
    For fixed non-negative integers $k$, $t$, and $n$ with $0 \leq t < k$, in $k_t$-Dyck paths of length $(k+1)n$ the 
    total number of down-steps $(1, -k)$ at a height of $i$ modulo $k$ is:
    \begin{enumerate} 
    \item \label{thm:part-1} if $k-t \leq i \leq k$, 
    \begin{equation*}
        \frac{(t+1)n+k+1}{n-1}\binom{(k+1)n+t}{n-2}.
    \end{equation*}
    \item \label{thm:part-2} if $1 \leq i \leq k-t-1$,
    \begin{equation*}
        (t+1)\binom{(k+1)n+t}{n-2}.
    \end{equation*}
    \end{enumerate}
\end{theorem}

\begin{proof}
    Let $F_t(\xbar)$ be defined as in the proof of Theorem~\ref{thm:main}. From~\eqref{eq:coeff-lagrange-2}, we have that
    \begin{equation*}
        [z^n] F_t(\xbar) = \frac{1}{n} [s^{n-1}] \sum_{j = k-t}^{k}\frac{x_j}{1-x_j s}\prod_{i=k-t}^{k}\frac{1}{(1-x_is)^{n+1}}
        \prod_{i=1}^{k-t-1}\frac{1}{(1-x_is)^n}.
    \end{equation*}
    Note that by determining the partial derivative of $F_{t}(\xbar)$ with 
    respect to $x_i$ and then setting $x_1 = x_2 = \cdots = x_k = 1$, we obtain a 
    generating function for the total number of down-steps $(1, -k)$ at a height 
    of $i$ modulo $k$ with respect to number of down-steps. Thus we begin by determining an expression for the 
    partial derivative of $F_{k-1}^{*}(\xbar)$ with respect to $x_i$, where $i$ is fixed and $1 \leq i \leq k$.
    So we differentiate with respect to $x_i$ and plug in $1$ for $x_1,\ldots,x_k$. For $i \leq k-t-1$, this yields
    \begin{align*}
        [z^n] \frac{\partial}{\partial x_i}  F_t(\xbar) \Big|_{x_1=\cdots=x_k = 1} &= \frac{1}{n} [s^{n-1}] \frac{t+1}{1-s} \cdot \frac{1}{(1-s)^{(t+1)(n+1)}} \cdot \frac{1}{(1-s)^{(k-t-2)n}} \cdot \frac{sn}{(1-s)^{n+1}} \\
            &= [s^{n-2}] \frac{(t+1)}{(1-s)^{kn+t+3}} \\
            &= (t+1) \binom{(k+1)n+t}{n-2}.
    \end{align*}
    For $i \geq k-t$, we can either use the same approach or make use of the fact that the number must be the same for all 
    such $i$ by symmetry, coupled with the fact that the total number of down-steps in all $k_t$-Dyck paths of length $(k+1)n$ is 
    \begin{equation*}
        n \cdot \frac{(t+1)}{(k+1)n+t+1} \binom{(k+1)n+t+1}{n} = (t+1) \binom{(k+1)n+t}{n-1}
    \end{equation*}
    by~\eqref{eq:raney}. Thus for $i \geq k-t$, the total number of down-steps at a height of $i$ modulo $k$ is
    \begin{equation*}
        \frac{(t+1) \binom{(k+1)n+t}{n-1} - (k-t-1) \cdot (t+1) \binom{(k+1)n+t}{n-2}}{t+1},
    \end{equation*}
    which simplifies to
    \begin{equation*}
        \frac{(t+1)n+k+1}{n-1} \binom{(k+1)n+t}{n-2}.
    \end{equation*}
\end{proof}
\begin{remark}
    Note that this proof could also be done by multiplying the formula in Theorem~\ref{thm:main} by $a_i$ and then summing over 
    all possible $k$-tuples $(a_1, a_2, \ldots ,a_k)$ and using the Vandermonde identity to simplify the summation.
\end{remark}

\begin{corollary}
    For fixed non-negative integers $k$, $t$, and $n$, with $0 \leq t < k$, the average number of  
    down-steps $(1, -k)$ at a height of $i$ modulo $k$ in $k_t$-Dyck paths of length $(k+1)n$ is:
    \begin{enumerate} 
    \item \label{cor:part-1} if $k-t \leq i \leq k$, 
    \begin{equation*}
        \frac{n((t+1)n+k+1)}{(kn+t+2)(t+1)}.
    \end{equation*}
    \item \label{cor:part-2} if $1 \leq i \leq k-t-1$,
    \begin{equation*}
        \frac{n(n-1)}{kn+t+2}.
    \end{equation*}
    \end{enumerate}
\end{corollary}

\begin{proof}
    This follows from Theorem~\ref{thm:total}, dividing each quantity by the total number of $k_t$-Dyck paths of 
    length $(k+1)n$ given in~\eqref{eq:raney}.
\end{proof}

\section{Related sequences}

In this section we use the phrase ``\textit{$k_t$-Dyck paths with $(a_1, \ldots, a_k)$}'' to mean the set of \mbox{$k_t$-Dyck} paths with down-step 
statistics at heights modulo $k$ of $(a_1, \ldots, a_k)$, where the $i$-th entry of the (1-indexed) tuple is the number of down-steps at a 
height of $i$ modulo $k$.

Note that there is a bijection between $k_t$-Dyck paths with $(a_1, \ldots, a_k)$ and $(k+1)_t$-Dyck 
paths with $(0, a_1, \ldots, a_k)$. This is because 
of a simple correspondence: inserting an up-step $(1, 1)$ before every up-step at a height of $k$ modulo $k$ and 
extending all down-steps to be $(1, -(k+1))$ converts 
a $k_t$-Dyck path with $(a_1, \ldots, a_k)$ to a $(k+1)_t$-Dyck path with $(0, a_1, \ldots, a_k)$, and the reverse 
follows plainly. This bijection can be extended inductively. Therefore we 
will state cases of interesting sequences in their simplest form, but analogous results can be extracted for higher values than $k$. Additionally, 
Lemma~\ref{lem:wag-comb-interp} provides a basis for reordering $a_i$'s. The following sequences of $(a_1, \ldots, a_k)$ give rise to interesting 
relationships between the height of down-steps modulo $k$ in $k_t$-Dyck paths and other combinatorial objects. Some of these relationships will be 
explored in the subsections that follow:
\begin{itemize}
    \item[\ref{subsec:peaks}] In this section three straightforward bijections between $k_t$-Dyck paths with a prescribed $(a_1, a_2)$
    and other combinatorial objects are given. The sequence OEIS~\href{https://oeis.org/A001700}{A001700} with 
    offset\footnote{By offset [value] we mean that the sequence in the OEIS is shifted [value] positions in order to obtain the sequence we refer 
    to. For example, the sequence $1, 1, 1, \ldots$ is the sequence $0, 0, 1, 1, 1, \ldots$ with offset 2.} 1 corresponds to $2_1$-Dyck 
    paths of length $3n$ with $(1, n-1)$. A bijection between these paths and the number of peaks in Dyck paths of length $2n$ is given in detail. 
    Additionally, a bijection between $2_0$-Dyck paths of length $3n$ with $(1, n-1)$ and the number of valleys in all Dyck paths of length $2n$ 
    is described. These objects have counting sequence OEIS~\href{https://oeis.org/A002054}{A002054} with offset 1. Finally, we show that $2_0$-Dyck 
    paths of length $3n$ with $(n-1, 1)$ are in bijection with Dyck paths of length $2n-2$, which are counted by the sequence 
    OEIS~\href{https://oeis.org/A000108}{A000108} (Catalan numbers) with offset 1.
    \item[\ref{subsec:2}] OEIS~\href{https://oeis.org/A002740}{A002740}, $(2, n-2)$ for $t = 0$ with offset 0. Simplifying this case in~\eqref{eq:main}, it 
    becomes $\frac{n-2}{2}\binom{2n-2}{n}$ for $n \geq 2$ and $0$ otherwise. According to the OEIS, this sequence counts the sum of the indices of the down-steps in all valleys 
    (a down-step followed immediately by an up-step) in Dyck paths of length $2n-2$.   
\end{itemize}
Several other special cases of~\eqref{eq:main} occur in the OEIS without combinatorial interpretations at present, including 
OEIS~\href{https://oeis.org/A110609}{A110609} with offset 1 for $3_1$-Dyck paths of length $3n$ with $(1, 1, n-2)$, and 
OEIS~\href{https://oeis.org/A188681}{A188681} with offset 0 for $2_1$-Dyck paths of length $6n$ with $(n, n)$. 

\subsection{Three straightforward bijective relationships between $2_t$-Dyck paths with a prescribed $(a_1, a_2)$
and other combinatorial objects}
\label{subsec:peaks}
There is a simple bijection between $2_1$-Dyck paths of length $3n$ with $(a_1, a_2) = (1, n-1)$ and peaks in Dyck paths of length $2n$. 
Note that other correspondences follow quite readily, for example, peaks in Dyck paths of length $2n$ under the so-called ``glove bijection''
translate into leaves in rooted ordered trees with $n$ edges. See Figure~\ref{fig:peak-mapping} for an example of the mapping.

The correspondence goes as follows: since there is exactly one down-step $(1, -2)$ at height $1$ modulo $2$, this step is preceded and 
followed by an up-step $(1, 1)$. This sequence of up-down-up can then be exchanged for a sequence of up-up-down, marking the down-step. 
In this way we have created a $2_0$-Dyck path which consists only of down-steps at a height of $0$ modulo $2$, and thus can reduce every 
two up-steps to one up-step $(1, 1)$, and every down-step $(1, -2)$ to a shorter down-step $(1, -1)$. The marked down-step is by construction 
preceded by an up-step, and uniquely marks a peak of the resulting Dyck path. This mapping is easily reversible. 

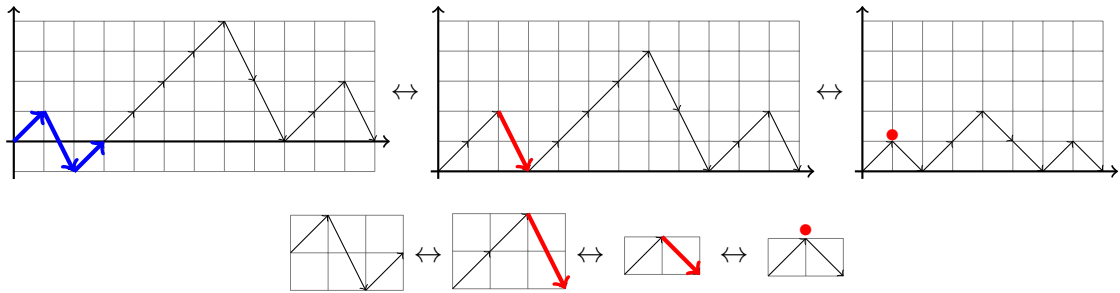
\begin{figure}[ht]
    \begin{tikzpicture}[scale=0.4]
        \draw[help lines] (0,-1) grid (12, 4);
        \draw[thick, ->] (-0.25, 0) -- (12.5, 0);
        \draw[thick, ->] (0, -1.25) -- (0, 4.5);
        \drawlatticepath{1, -2, 1, 1, 1, 1, 1, -2, -2, 1, 1, -2}
        \draw[->, ultra thick, blue] (0, 0) to (1, 1);
        \draw[->, ultra thick, blue] (1, 1) to (2, -1);
        \draw[->, ultra thick, blue] (2, -1) to (3, 0);
    \end{tikzpicture}\raisebox{2.8em}{$\leftrightarrow$}
    \begin{tikzpicture}[scale=0.4]
        \draw[help lines] (0,0) grid (12, 5);
        \draw[thick, ->] (-0.25, 0) -- (12.5, 0);
        \draw[thick, ->] (0, -0.25) -- (0, 5.5);
        \drawlatticepath{1, 1, -2, 1, 1, 1, 1, -2, -2, 1, 1, -2}
        \draw[->, ultra thick, red] (2, 2) to (3, 0);
    \end{tikzpicture}\raisebox{2.8em}{$\leftrightarrow$}
    \begin{tikzpicture}[scale=0.4]
        \draw[help lines] (0,0) grid (8, 5);
        \draw[thick, ->] (-0.25, 0) -- (8.5, 0);
        \draw[thick, ->] (0, -0.25) -- (0, 5.5);
        \drawlatticepath{1, -1, 1, 1, -1, -1, 1, -1}
        \node[red] at (1, 1.2) {$\bullet$};
    \end{tikzpicture}
    \vspace{1em}

    \begin{tikzpicture}[scale=0.5]
        \draw[help lines] (0,-1) grid (3,1);
        \drawlatticepath{1, -2, 1}
    \end{tikzpicture} \raisebox{1em}{$\leftrightarrow$}
    \begin{tikzpicture}[scale=0.5]
        \draw[help lines] (0, 0) grid (3,2);
        \drawlatticepath{1, 1, -2}
        \draw[->, ultra thick, red] (2, 2) to (3, 0);
    \end{tikzpicture}\;\raisebox{1em}{$\leftrightarrow$}
    \raisebox{0.5em}{
    \begin{tikzpicture}[scale=0.5]
        \draw[help lines] (0,0) grid (2,1);
        \drawlatticepath{1, -1}
        \draw[->, ultra thick, red] (1, 1) to (2, 0);
    \end{tikzpicture}}
    \;\raisebox{1em}{$\leftrightarrow$}
    \raisebox{0.5em}{
    \begin{tikzpicture}[scale=0.5]
        \draw[help lines] (0,0) grid (2,1);
        \drawlatticepath{1, -1}
        \node[red] at (1, 1.2) {$\bullet$};
    \end{tikzpicture}}

    \caption{The mapping between a $2_1$-Dyck path of length $12$ with $(a_1, a_2) = (1, 3)$ and a peak in a Dyck path of length $8$.}
    \label{fig:peak-mapping}
\end{figure}

A similar bijection holds between OEIS~\href{https://oeis.org/A002054}{A002054} with offset 1, which is the sequence for the number of valleys in
all Dyck paths of length $2n$, and $2_0$-Dyck paths of length $3n$ with $(1, n-1)$. The down-step $(1, -2)$ is shifted one up-step $(1, 1)$ 
to the left, creating a sequence down-up-up, which can be marked as a valley. Then `reduce' pairs of up-steps and down-steps as described 
previously to get a Dyck path with a marked valley. 

The final closely related bijection relates $2_0$-Dyck paths of length $3n$ with $(n-1, 1)$ and Dyck paths of length $2n-2$ 
(OEIS~\href{https://oeis.org/A000108}{A000108} with offset 1). The single down-step
at height $0$ modulo $2$ must be at the end of the path for it to end on the $x$-axis. We remove the first up-step in the path, and also remove the 
final up-step and down-step (the final up-step is immediately before the final down-step), to obtain a path of length $3n-3$ with all down-steps at 
a height of $0$ modulo $2$. Again `reduce' pairs of up-steps and down-steps to obtain a Dyck path of length $2n-2$. 

\subsection{The family of $2_0$-Dyck paths of length $3n$ with $(a_1, a_2) = (2, n-2)$ is enumerated by the sum of indices of the down-steps 
in all valleys in Dyck paths of length $2(n-1)$}
\label{subsec:2}
Since the sum of indices of the down-steps in all valleys (down-steps followed by up-steps) in Dyck paths of length $2n-2$ is being considered, 
we represent the total sum using double-marked Dyck paths of length $2n-2$. That is, if a Dyck path has a valley whose down-step occurs at 
index $i$, then we mark the valley in question, and create $i$ copies of the path where exactly one (unique) step in the path is marked in each 
copy. See Figure~\ref{fig:markings} below for an example of this marking. 

\begin{figure}[ht]
    \begin{tikzpicture}[scale = 0.4]
        \draw[help lines] (0, 0) grid (8, 3);
        \draw[->, thick] (-0.5, 0) to (8, 0);
        \draw[->, thick] (0, -0.5) to (0, 3);

        \drawlatticepath{1, 1, -1, 1, -1, -1, 1, -1}
        \draw[->, ultra thick, red] (0, 0) to (1, 1);

        \node at (3, 0.7) {$\bullet$}; 
    \end{tikzpicture}
    \begin{tikzpicture}[scale = 0.4]
        \draw[help lines] (0, 0) grid (8, 3);
        \draw[->, thick] (-0.5, 0) to (8, 0);
        \draw[->, thick] (0, -0.5) to (0, 3);

        \drawlatticepath{1, 1, -1, 1, -1, -1, 1, -1}
        \draw[->, ultra thick, red] (1, 1) to (2, 2);

        \node at (3, 0.7) {$\bullet$}; 
    \end{tikzpicture}
    \begin{tikzpicture}[scale = 0.4]
        \draw[help lines] (0, 0) grid (8, 3);
        \draw[->, thick] (-0.5, 0) to (8, 0);
        \draw[->, thick] (0, -0.5) to (0, 3);

        \drawlatticepath{1, 1, -1, 1, -1, -1, 1, -1}
        \draw[->, ultra thick, red] (2, 2) to (3, 1);

        \node at (3, 0.7) {$\bullet$}; 
    \end{tikzpicture}

    \begin{tikzpicture}[scale = 0.4]
        \draw[help lines] (0, 0) grid (8, 3);
        \draw[->, thick] (-0.5, 0) to (8, 0);
        \draw[->, thick] (0, -0.5) to (0, 3);

        \drawlatticepath{1, 1, -1, 1, -1, -1, 1, -1}
        \draw[->, ultra thick, red] (0, 0) to (1, 1);

        \node at (6, -0.35) {$\bullet$}; 
    \end{tikzpicture}
    \begin{tikzpicture}[scale = 0.4]
        \draw[help lines] (0, 0) grid (8, 3);
        \draw[->, thick] (-0.5, 0) to (8, 0);
        \draw[->, thick] (0, -0.5) to (0, 3);

        \drawlatticepath{1, 1, -1, 1, -1, -1, 1, -1}
        \draw[->, ultra thick, red] (1, 1) to (2, 2);

        \node at (6, -0.35) {$\bullet$}; 
    \end{tikzpicture}
    \begin{tikzpicture}[scale = 0.4]
        \draw[help lines] (0, 0) grid (8, 3);
        \draw[->, thick] (-0.5, 0) to (8, 0);
        \draw[->, thick] (0, -0.5) to (0, 3);

        \drawlatticepath{1, 1, -1, 1, -1, -1, 1, -1}
        \draw[->, ultra thick, red] (2, 2) to (3, 1);

        \node at (6, -0.35) {$\bullet$}; 
    \end{tikzpicture}
    \begin{tikzpicture}[scale = 0.4]
        \draw[help lines] (0, 0) grid (8, 3);
        \draw[->, thick] (-0.5, 0) to (8, 0);
        \draw[->, thick] (0, -0.5) to (0, 3);

        \drawlatticepath{1, 1, -1, 1, -1, -1, 1, -1}
        \draw[->, ultra thick, red] (3, 1) to (4, 2);

        \node at (6, -0.3) {$\bullet$}; 
    \end{tikzpicture}
    \begin{tikzpicture}[scale = 0.4]
        \draw[help lines] (0, 0) grid (8, 3);
        \draw[->, thick] (-0.5, 0) to (8, 0);
        \draw[->, thick] (0, -0.5) to (0, 3);

        \drawlatticepath{1, 1, -1, 1, -1, -1, 1, -1}
        \draw[->, ultra thick, red] (4, 2) to (5, 1);

        \node at (6, -0.35) {$\bullet$}; 
    \end{tikzpicture}
    \begin{tikzpicture}[scale = 0.4]
        \draw[help lines] (0, 0) grid (8, 3);
        \draw[->, thick] (-0.5, 0) to (8, 0);
        \draw[->, thick] (0, -0.5) to (0, 3);

        \drawlatticepath{1, 1, -1, 1, -1, -1, 1, -1}
        \draw[->, ultra thick, red] (5, 1) to (6, 0);

        \node at (6, -0.35) {$\bullet$}; 
    \end{tikzpicture}
    \caption{The sum of indices of the down-steps in the valleys of the Dyck path given by step-sequence $1, 1, -1, 1, -1, -1, 1, -1$ is $9$, 
    and the $9$ double-marked Dyck paths which represent this are drawn, with valley-markings (circle) and unique step-markings (red, thick) depicted.}
    \label{fig:markings}
\end{figure}
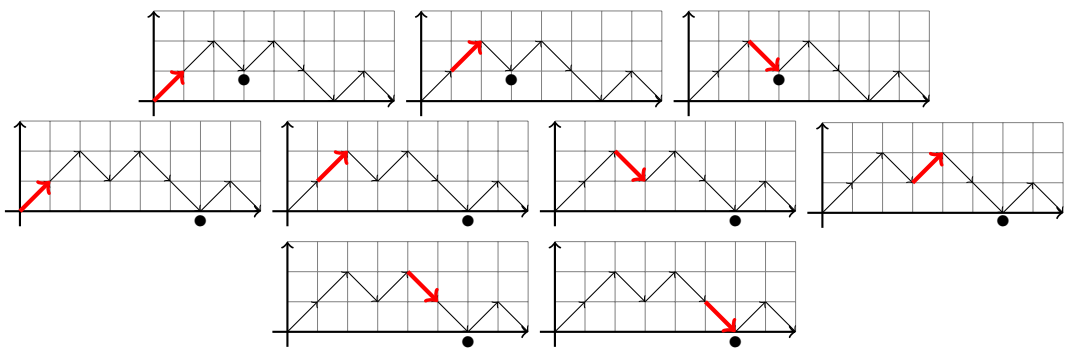
In this section we will provide a bijection between $2_0$-Dyck paths with $(2, n-2)$ and these double-marked Dyck paths of length $2n-2$.

Consider an arbitrary $2_0$-Dyck path of length $3n$ with $(2, n-2)$. Since the number of down-steps at a height of $1$ modulo $2$ is fixed 
at $2$, there are two possibilities for the relative positioning of these two down-steps: (A) they are separated by at least two up-steps, 
and (B) they are adjacent in the path. To avoid repetition, we refer to the two down-steps at a height of 1 modulo 2 as ``odd down-steps''. 
The steps of the mapping to double-marked Dyck paths are given as follows, where examples of the different cases can be seen in  
Figures~\ref{fig:bij-mapping-B},~\ref{fig:bij-mapping-A1},~and~\ref{fig:bij-mapping-A2}.
The reasons for each step being permissible is given in italics: 
\begin{enumerate}
    \item Mark the two odd down-steps. If case A holds, shift the rightmost odd down-step in the path left one 
    up-step. If case B holds, shift both odd down-steps to the left one up-step. 

    \noindent \textit{At this point we have a $2_0$-Dyck path of length $3n$ with either one odd down-step remaining (case A), 
    or no odd down-steps remaining (case B). The left shift is possible in both cases, as the odd down-steps must start at a height of 
    $1$ or above, and so a left shift of one up-step will stay weakly above the $x$-axis, and there will always be at least one 
    up-step between odd down-steps and the other down-steps so a shift by one will not change the order of occurrence of down-steps. 
    } 
    \item Mark the endpoint of the rightmost (previously) odd down-step with a circle, and remove its previous marking. This new marking 
    will be the \emph{valley-marking}. 
    
    \noindent \textit{Since the rightmost (previously) odd down-step is always followed by an up-step, the shift to the left one up-step
    has created the step-sequence down-up-up, thus introducing a valley.}
    \item If case A held (e.g.\ Figure~\ref{fig:bij-mapping-B}), then remove the leftmost (unshifted) odd down-step as well as the up-steps immediately to the left and right 
    of it, marking the step to the left of the removed steps. This is the \emph{step-marking}. 

    \noindent \textit{The step-sequence removed is up-down-up, and thus no height shift takes place and all conditions of a valid 
    $2_0$-Dyck path are still met. The marking stores where this step was removed, and removing this step removes the last odd 
    down-step left in the path.}
    \item If case B held, we consider a return to the $x$-axis to be any point in the path where the endpoint of a step is on the 
    $x$-axis, and there are two possibilities:
        \begin{enumerate}
            \item The section of the path to the left of the leftmost odd down-step contains no returns (e.g.\ Figure~\ref{fig:bij-mapping-A1}). 
            In this case, remove the first two up-steps in the path and the leftmost (previously) odd down-step, and mark the (previously) 
            odd down-step in the valley with a step-marking.
            
            \noindent \textit{Such a section of path which contains no returns and has all down-steps at a height of 0 modulo 2 
            must begin with a sequence of at least four consecutive up-steps -- an odd number would not be possible because of the 
            condition on the down-steps, and a maximal sequence of two consecutive up-steps would result in a return. Therefore removing two 
            up-steps at the start of the path and a down-step before another return of the path would not cause the path to go below the $x$-axis.  
            Since the down-step in the valley cannot be preceded by an up-down-up as in the case above, this provides a way to mark this step.}
            \item The section of the path to the left of the leftmost (previously) odd down-step contains at least one return (e.g.\ Figure~\ref{fig:bij-mapping-A2}). In this case, give the 
            down-step immediately before the rightmost return to the left of the 
            leftmost (previously) odd down-step the step-marking, and remove the two up-steps to the right of this return. Additionally, 
            remove the (previously) leftmost odd down-step.

            \noindent \textit{In this case the first return to the left must be a down-step followed by 
            at least four up-steps, where we can make the required removals as before.}
        \end{enumerate}
    \item What we now have is a $2_0$-Dyck path of length $3n-3$ where all down-steps are at a height of $0$ modulo $2$. From left to right in 
    the path we reduce every two up-steps to one up-step $(1, 1)$, and every down-step $(1, -2)$ to a down-step $(1, -1)$, step-marking the 
    resulting step if one of the steps to be reduced has the step-marking. We also preserve the relative position of the valley-marking.
\end{enumerate}

\begin{figure}[ht]

    \begin{tikzpicture}[scale = 0.3]
        \draw [help lines] (0, 0) grid (15, 6);
        \draw [->, thick] (-0.5, 0) to (15.5, 0);
        \draw [->, thick] (0, -0.5) to (0, 6.5);
        
        \drawlatticepath{1, 1, 1, 1, 1, 1, -2, -2, 1, -2, 1, 1, -2, 1, -2}
    \end{tikzpicture}\raisebox{2em}{$\xrightarrow{\text{(1)}}$}
    \begin{tikzpicture}[scale = 0.3]
        \draw [help lines] (0, 0) grid (15, 6);
        \draw [->, thick] (-0.5, 0) to (15.5, 0);
        \draw [->, thick] (0, -0.5) to (0, 6.5);
        
        \drawlatticepath{1, 1, 1, 1, 1, 1, -2, -2, 1, -2, 1, -2, 1, 1, -2}

        \draw [->, ultra thick] (9, 3) to (10, 1);
        \draw [->, ultra thick] (11, 2) to (12, 0);   
    \end{tikzpicture}
    
    \raisebox{2em}{$\xrightarrow{\text{(2)}}$}
    \begin{tikzpicture}[scale = 0.3]
        \draw [help lines] (0, 0) grid (15, 6);
        \draw [->, thick] (-0.5, 0) to (15.5, 0);
        \draw [->, thick] (0, -0.5) to (0, 6.5);
        
        \drawlatticepath{1, 1, 1, 1, 1, 1, -2, -2, 1, -2, 1, -2, 1, 1, -2}

        \draw [->, ultra thick] (9, 3) to (10, 1);
        \node at (12, -0.5) {$\bullet$};
    \end{tikzpicture}\raisebox{2em}{$\xrightarrow{\text{(3)}}$}
    \begin{tikzpicture}[scale = 0.3]
        \draw [help lines] (0, 0) grid (12, 6);
        \draw [->, thick] (-0.5, 0) to (12.5, 0);
        \draw [->, thick] (0, -0.5) to (0, 6.5);
        
        \drawlatticepath{1, 1, 1, 1, 1, 1, -2, -2, -2, 1, 1, -2}

        \draw [->, red, ultra thick] (7, 4) to (8, 2);
        \node at (9, -0.5) {$\bullet$};
    \end{tikzpicture}
    \raisebox{2em}{$\xrightarrow{\text{(5)}}$}
    \begin{tikzpicture}[scale = 0.3]
        \draw [help lines] (0, 0) grid (8, 3);
        \draw [->, thick] (-0.5, 0) to (8.5, 0);
        \draw [->, thick] (0, -0.5) to (0, 3.5);
        
        \drawlatticepath{1, 1, 1, -1, -1, -1, 1, -1}

        \draw [->, red, ultra thick] (4, 2) to (5, 1);
        \node at (6, -0.5) {$\bullet$};
    \end{tikzpicture}

    \caption{An example of the mapping in Case A: where the odd down-steps are separated by at least two up-steps in the path.}
    \label{fig:bij-mapping-B}
\end{figure}
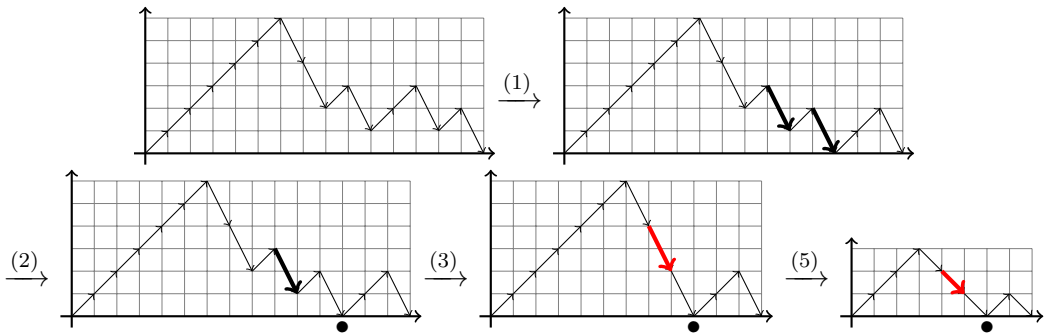
\begin{figure}[ht]

    \begin{tikzpicture}[scale = 0.3]
        \draw [help lines] (0, 0) grid (15, 6);
        \draw [->, thick] (-0.5, 0) to (15.5, 0);
        \draw [->, thick] (0, -0.5) to (0, 6.5);
        
        \drawlatticepath{1, 1, 1, 1, 1, 1, -2, 1, -2, -2, 1, -2, 1, 1, -2}
    \end{tikzpicture}\raisebox{2em}{$\xrightarrow{\text{(1)}}$}
    \begin{tikzpicture}[scale = 0.3]
        \draw [help lines] (0, 0) grid (15, 6);
        \draw [->, thick] (-0.5, 0) to (15.5, 0);
        \draw [->, thick] (0, -0.5) to (0, 6.5);
        
        \drawlatticepath{1, 1, 1, 1, 1, 1, -2, -2, -2, 1, 1, -2, 1, 1, -2}

        \draw [->, ultra thick] (7, 4) to (8, 2);
        \draw [->, ultra thick] (8, 2) to (9, 0);   
    \end{tikzpicture}
    
    \raisebox{2em}{$\xrightarrow{\text{(2)}}$}
    \begin{tikzpicture}[scale = 0.3]
        \draw [help lines] (0, 0) grid (15, 6);
        \draw [->, thick] (-0.5, 0) to (15.5, 0);
        \draw [->, thick] (0, -0.5) to (0, 6.5);
        
        \drawlatticepath{1, 1, 1, 1, 1, 1, -2, -2, -2, 1, 1, -2, 1, 1, -2}

        \draw [->, ultra thick] (7, 4) to (8, 2);
        \node at (9, -0.5) {$\bullet$};
    \end{tikzpicture}\raisebox{2em}{$\xrightarrow{\text{(4a)}}$}
    \begin{tikzpicture}[scale = 0.3]
        \draw [help lines] (0, 0) grid (12, 6);
        \draw [->, thick] (-0.5, 0) to (12.5, 0);
        \draw [->, thick] (0, -0.5) to (0, 6.5);
        
        \drawlatticepath{1, 1, 1, 1, -2, -2, 1, 1, -2, 1, 1, -2}

        \draw [->, red, ultra thick] (5, 2) to (6, 0);
        \node at (6, -0.5) {$\bullet$};
    \end{tikzpicture}
    \raisebox{2em}{$\xrightarrow{\text{(5)}}$}
    \begin{tikzpicture}[scale = 0.3]
        \draw [help lines] (0, 0) grid (8, 3);
        \draw [->, thick] (-0.5, 0) to (8.5, 0);
        \draw [->, thick] (0, -0.5) to (0, 3.5);
        
        \drawlatticepath{1, 1, -1, -1, 1, -1, 1, -1}

        \draw [->, red, ultra thick] (3, 1) to (4, 0);
        \node at (4, -0.5) {$\bullet$};
    \end{tikzpicture}

    \caption{An example of the mapping in Case B: where the odd down-steps are adjacent 
    and there are no returns to the left of the leftmost odd down-step in the path.}
    \label{fig:bij-mapping-A1}
\end{figure}
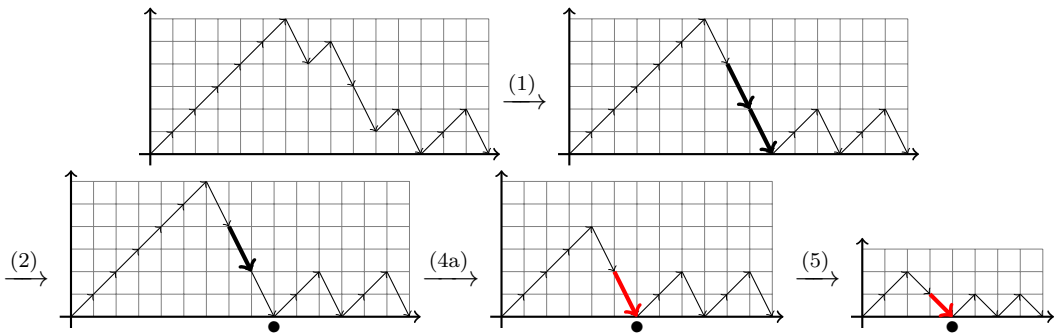
\begin{figure}[ht]

    \begin{tikzpicture}[scale = 0.3]
        \draw [help lines] (0, 0) grid (15, 6);
        \draw [->, thick] (-0.5, 0) to (15.5, 0);
        \draw [->, thick] (0, -0.5) to (0, 6.5);
        
        \drawlatticepath{1, 1, -2, 1, 1, -2, 1, 1, 1, 1, 1, -2, -2, 1, -2}
    \end{tikzpicture}\raisebox{2em}{$\xrightarrow{\text{(1)}}$}
    \begin{tikzpicture}[scale = 0.3]
        \draw [help lines] (0, 0) grid (15, 6);
        \draw [->, thick] (-0.5, 0) to (15.5, 0);
        \draw [->, thick] (0, -0.5) to (0, 6.5);
        
        \drawlatticepath{1, 1, -2, 1, 1, -2, 1, 1, 1, 1, -2, -2, 1, 1, -2}

        \draw [->, ultra thick] (10, 4) to (11, 2);
        \draw [->, ultra thick] (11, 2) to (12, 0);   
    \end{tikzpicture}
    
    \raisebox{2em}{$\xrightarrow{\text{(2)}}$}
    \begin{tikzpicture}[scale = 0.3]
        \draw [help lines] (0, 0) grid (15, 6);
        \draw [->, thick] (-0.5, 0) to (15.5, 0);
        \draw [->, thick] (0, -0.5) to (0, 6.5);
        
        \drawlatticepath{1, 1, -2, 1, 1, -2, 1, 1, 1, 1, -2, -2, 1, 1, -2}

        \draw [->, ultra thick] (10, 4) to (11, 2);
        \node at (12, -0.5) {$\bullet$};
        \node at (3, -0.5) {$\circ$};
        \node at (6, -0.5) {$\circ$};
        \node at (0, -0.5) {$\circ$};
    \end{tikzpicture}\raisebox{2em}{$\xrightarrow{\text{(4b)}}$}
    \begin{tikzpicture}[scale = 0.3]
        \draw [help lines] (0, 0) grid (12, 6);
        \draw [->, thick] (-0.5, 0) to (12.5, 0);
        \draw [->, thick] (0, -0.5) to (0, 6.5);
        
        \drawlatticepath{1, 1, -2, 1, 1, -2, 1, 1, -2, 1, 1, -2}

        \draw [->, red, ultra thick] (5, 2) to (6, 0);
        \node at (9, -0.5) {$\bullet$};
        
    \end{tikzpicture}
    \raisebox{2em}{$\xrightarrow{\text{(5)}}$}
    \begin{tikzpicture}[scale = 0.3]
        \draw [help lines] (0, 0) grid (8, 3);
        \draw [->, thick] (-0.5, 0) to (8.5, 0);
        \draw [->, thick] (0, -0.5) to (0, 3.5);
        
        \drawlatticepath{1, -1, 1, -1, 1, -1, 1, -1}

        \draw [->, red, ultra thick] (3, 1) to (4, 0);
        \node at (6, -0.5) {$\bullet$};
    \end{tikzpicture}

    \caption{An example of the mapping in Case B: where the odd down-steps are adjacent 
    and there is at least one return to the left of the leftmost odd down-step in the path. 
    Returns to the left of the leftmost odd down-step are marked with an empty circle.}
    \label{fig:bij-mapping-A2}
\end{figure}
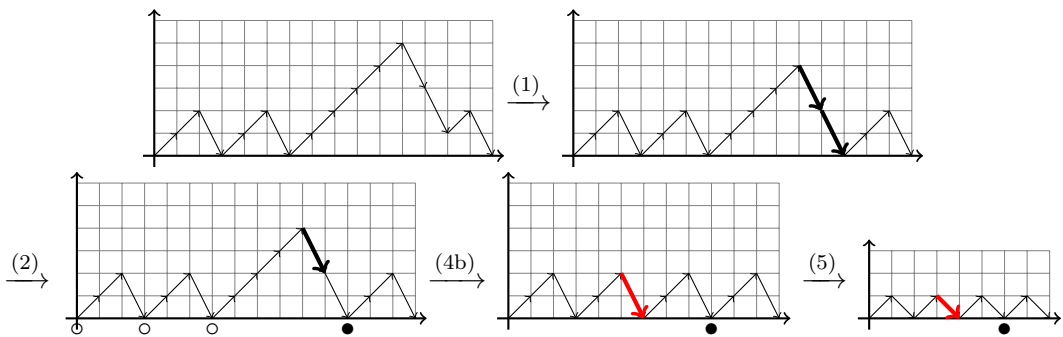

From the description of the mapping, it is clear that a $2_0$-Dyck path with $(2, n-2)$ is mapped to a double-marked Dyck path. 
We now prove bijectivity.

It is clear that each of the three possible cases (described in bullets 3, 4a, 4b) for the mapping is independently 
injective, what remains to show is that no pair of cases yields the same double-marked Dyck path. 
\begin{itemize}
    \item \textbf{Case 3 and 4a:} In Case 3 the step to the left of a step-sequence of up-down-up becomes marked, 
    where the down-step in that sequence is the leftmost odd down-step. However, in Case 4a the rightmost (previously) 
    odd down-step, which is in the marked valley, becomes marked. 
    It is therefore not possible to mark the same step under both mappings.
    \item \textbf{Case 3 and 4b:} The characteristic property of the resulting path from Case 4b is that the marked step 
    is a down-step which ends on the $x$-axis. It is not possible for a path satisfying Case 3 to result in such 
    a path, as in order for this to happen, the step-sequence up-down-up must occur immediately to the right of a return to the 
    $x$-axis, which would violate the condition that a $2_0$-Dyck path stays weakly above the $x$-axis. 
    \item \textbf{Case 4a and 4b:} In Case 4a the down-step in the marked valley is marked, and by definition of Case 4b, the 
    marked down-step occurs strictly to the left of the down-step associated with the valley. 
\end{itemize}
Therefore we have established injectivity of our function. Now, to establish surjectivity, we obtain the reverse mapping as 
(briefly) follows:
\begin{enumerate}
    \item Exchange every up-step $(1, 1)$ for two up-steps $(1, 1)$, marking the second of the two if the original step was 
    marked, and exchange down-steps $(1, -1)$ for down-steps $(1, -2)$, marking the down-step if it was originally marked. 
    \item Consider the marked step which results from (1):
    \begin{itemize}
        \item If it is in the marked valley, reverse the procedure described in (4a) -- add two up-steps to the start
        of the path and a down-step immediately to the left of the marked step. 
        \item Else if it is a down-step which has an endpoint on the $x$-axis, reverse the procedure described in (4b) -- add 
        two up-steps immediately to the right of the marked step, and one down-step immediately to the left of the down-step 
        in the valley. 
        \item Else insert a step-sequence of up-down-up immediately to the right of the marked step.  
    \end{itemize}
    Mark the newly inserted down-step.
    \item Remove the valley-marking and mark the down-step in the previously marked valley. Shift the rightmost marked 
    down-step and any adjacent marked down-steps one up-step to the right. Remove the markings. What results is a $2_0$-Dyck path 
    that satisfies our conditions.  
\end{enumerate}

\bibliographystyle{plain}
\bibliography{bib/cheub}

\end{document}